\documentclass[12pt]{amsart}
\usepackage{amscd,amsmath,amssymb,enumerate,amsfonts,mathrsfs,txfonts}
\usepackage{comment}
\usepackage{hyperref}
\usepackage{xcolor}
\usepackage{tikz,pgfplots,pgf}
\usepackage[all]{xy}

\tikzset{node distance=3cm, auto}
\usetikzlibrary{calc} 
\usetikzlibrary{trees} 

\newtheorem{theorem}{Theorem}[section]
\newtheorem{proposition}[theorem]{Proposition}
\newtheorem{definition}[theorem]{Definition}
\newtheorem{corollary}[theorem]{Corollary}

\def\AS{\mathcal{AS}}
\def\At{\mathrm{At}}
\def\Gv{\G_{v}}

\def\Du{\mathcal{D}}

\def\G{\mathcal{G}}

\def\H{\mathcal{H}}
\def\Hv{\mathcal{H}_v}
\def\I{\mathcal{I}}
\def\J{\mathcal{J}}
\def\L{\mathcal{L}}
\def\F{\mathcal{F}}
\def\K{\mathcal{K}}

\def\W{\mathcal{W}}

\def\Nu{\mathcal{N}}

\def\R{\mathcal{R}}
\def\S{\mathcal{S}}

\def\C{\mathbb{C}}

\def\N{\mathbb{N}}

\def\abco{\mathrm{aco}}

\def\lin{\mathrm{lin}}

\def\d{\mathrm{dual}}

\usepackage{vmargin}
\setpapersize{A4}
\setmargins{2.5cm}      
{1.5cm}                        
{16.5cm}                      
{23.42cm}                   
{10pt}                          
{1cm}                           
{0pt}                          
{2cm}

\begin{document}


\title[The injective hull of ideals of weighted holomorphic mappings]{The injective hull of ideals of weighted holomorphic mappings}

\author[A. Jim{\'e}nez-Vargas]{A. Jim{\'e}nez-Vargas}
\address[A. Jim{\'e}nez-Vargas]{Departamento de Matem\'aticas, Universidad de Almer\'ia, Ctra. de Sacramento s/n, 04120 La Ca\~{n}ada de San Urbano, Almer\'ia, Spain.}
\email{ajimenez@ual.es}

\author[M.I. Ram\'{i}rez]{Mar{\'\i}a Isabel Ram{\'\i}rez}
\address[M.I. Ram\'{i}rez]{Departamento de Matem\'aticas, Universidad de Almer\'ia, Ctra. de Sacramento s/n, 04120 La Ca\~{n}ada de San Urbano, Almer\'ia, Spain.}
\email{mramirez@ual.es}

\date{\today}



\begin{abstract}
We study the injectivity of normed ideals of weighted holomorphic mappings. To be more precise, the concept of injective hull of normed weighted holomorphic ideals is introduced and characterized in terms of a domination property. The injective hulls of those ideals -- generated by the procedures of composition and dual -- are described and these descriptions are applied to some examples of such ideals. A characterization of the closed injective hull of an operator ideal in terms of an Ehrling-type inequality -- due to Jarchow and Pelczy\'nski-- is established for weighted holomorphic mappings.
\end{abstract}
\maketitle


\section*{Introduction}

Influenced by the concept of operator ideals (see the book \cite{Pie-80} by Pietsch), the notion of ideals of weighted holomorphic mappings was introduced in \cite{CabJimKet-24}, although also the ideals of bounded holomorphic mappings were analysed in  \cite{CabJimRui-23}. In \cite{CabJimKet-24}, the composition procedure to generate weighted holomorphic ideals was studied and some examples of such ideals were presented. 

Our aim in this paper is to address the injective procedure in the context of weighted holomorphic mappings. In the linear setting, the concept of injective hull of an operator ideal was dealt by Pietsch \cite{Pie-80}, although some ingredients already appeared in the paper \cite{Ste-70} by Stephani. 

Given an open subset $U$ of a complex Banach space $E$, a weight $v$ on $U$ is a (strictly) positive continuous function. For any complex Banach space $F$, let $\H(U,F)$ be the space of all holomorphic mappings from $U$ into $F$. The space of weighted holomorphic mappings, $\H_v^\infty(U,F)$, is the Banach space of all mappings $f\in\H(U,F)$ so that   
$$
\left\|f\right\|_v:=\sup\left\{v(x)\left\|f(x)\right\|\colon x\in U\right\}<\infty ,
$$ 
under the weighted supremum norm $\left\|\cdot\right\|_v$. We will write $\H^\infty_v(U)$ instead of $\H^\infty_v(U,\C)$. 

About the theory of weighted holomorphic mappings, the interested reader can consult the papers \cite{BieSum-93} by Bierstedt and Summers, \cite{BonDomLin-99,BonDomLin-01} by Bonet, Domanski and Lindstr\"om, and \cite{GupBaw-16} by Gupta and Baweja. See also the recent survey \cite{Bon-22} by Bonet on these function spaces, and the references therein.

By definition, the injective hull of a normed weighted holomorphic ideal $[\I^{\Hv^\infty},\|\cdot\|_{\I^{\Hv^\infty}}]$ is the smallest injective normed weighted holomorphic ideal containing $\I^{\Hv^\infty}$. In Subsection \ref{subsection 1}, we will establish the existence of this injective hull, and -- as a immediate consequence -- the injectivity of a normed weighted holomorphic ideal is characterized by the coincidence with its injective hull. 

In Subsection \ref{subsection 2}, a characterization of the injective hull of a normed weighted holomorphic ideal is stated by means of a domination property, and it is applied to describe the injectivity of a normed weighted holomorphic ideal in a form similar to those obtained in the linear and polynomial versions \cite{BotCamSan-17,BotTor-20}. 

Using the linearization of weighted holomorphic mappings, we describe in Subsection \ref{subsection 3} the injective hull of composition ideals of weighted holomorphic mappings and apply this description to establish the injectivity of the normed weighted holomorphic ideals generated by composition with some distinguished classes of bounded linear operators such as finite-rank, compact, weakly compact, separable, Rosenthal and Asplund operators. 

In Subsection \ref{subsection 4}, the concept of dual weighted holomorphic ideal of an operator ideal $\I$ is introduced and showed that it coincides with the weighted holomorphic ideal generated by composition with the dual operator ideal $\I^\d$. Moreover, we study the injectivity of such dual weighted holomorphic ideals as well as the dual weighted holomorphic ideals of the ideals of $p$-compact and Cohen strongly $p$-summing operators for any $p\in (1,\infty)$.

Subsection \ref{subsection 5} presents a weighted holomorphic variant of a characterization --due to Jarchow and Pelczy\'nski \cite{Jar-81}-- of the closed injective hull of an operator ideal by means of an Ehrling-type inequality \cite{Ehr-54}.

It should be noted that different authors have studied these questions for ideals of functions in both linear settings (for classical $p$-compact operators \cite{Fou-18}, $(p,q)$-compact operators \cite{Kim-19}, weakly $p$-nuclear operators \cite{Kim-19b} and multilinear mappings \cite{ManRuedSan-20}) as well as in non-linear contexts (for holomorphic mappings \cite{GonGut-99}, polynomials \cite{BotTor-20} and Lipschitz operators \cite{AchDahTur-20}), among others.


\section{Results}\label{section 2}

We will present the results of this paper in various subsections. From now on, unless otherwise stated, $E$ will denote a complex Banach space, $U$ an open subset of $E$, $v$ a weight on $U$, and $F$ a complex Banach space.

Our notation is standard. $\L(E,F)$ denotes the Banach space of all bounded linear operators from $E$ into $F$, equipped with the operator canonical norm. $E^{*}$ and $B_E$ represent the dual space and the closed unit ball of $E$, respectively. Given a set $A\subseteq E$, $\overline{\lin}(A)$ and $\overline{\abco}(A)$ stand for the norm closed linear hull and the norm closed absolutely convex hull of $A$ in $E$. 


\subsection{The injective hull of ideals of weighted holomorphic mappings}\label{subsection 1}

In the light of Definition 2.4 in \cite{CabJimKet-24}, a normed (Banach) ideal of weighted holomorphic mappings -- or, in short,  a normed (Banach) weighted holomorphic ideal -- is an assignment $[\I^{\Hv^\infty},\left\|\cdot\right\|_{\I^{\Hv^\infty}}]$ which associates every pair $(U,F)$, -- where $E$ is a complex Banach space, $U$ is an open subset of $E$ and $F$ is a complex Banach space-- to both a set $\I^{\Hv^\infty}(U,F)\subseteq\Hv^\infty(U,F)$ and a function $\|\cdot\|_{\I^{\Hv^\infty}}\colon\I^{\Hv^\infty}(U,F)\to\mathbb{R}$ satisfying
\begin{itemize}
\item[(P1)] $(\I^{\Hv^\infty}(U,F),\|\cdot\|_{\I^{\Hv^\infty}})$ is a normed (Banach) space with $\|f\|_v\leq \|f\|_{\I^{\Hv^\infty}}$ for $f\in\I^{\Hv^\infty}(U,F)$, 
\item[(P2)] Given $h\in\Hv^\infty(U)$ and $y\in F$, the map $h\cdot y\colon x\in U\mapsto h(x)y\in F$ is in $\I^{\Hv^\infty}(U,F)$ with $\|h\cdot y\|_{\I^{\Hv^\infty}}=\|h\|_v||y||$,
\item[(P3)] The ideal property: if $V$ is an open subset of $E$ such that $V\subseteq U$, $h\in\H(V,U)$ with $c_v(h):=\sup_{x\in V}(v(x)/v(h(x)))<\infty$, $f\in\I^{\Hv^\infty}(U,F)$ and $T\in\L(F,G)$ where $G$ is a complex Banach space, then $T\circ f\circ h\in\I^{\Hv^\infty}(V,G)$ with $\|T\circ f\circ h\|_{\I^{\Hv^\infty}}\leq\left\|T\right\|\|f\|_{\I^{\Hv^\infty}}c_v(h)$.
\end{itemize}

According to Sections 4.6 and 8.4 in \cite{Pie-80}, an operator ideal $\I$ is said to be injective if for each Banach space $G$ and each isometric linear embedding $\iota\colon F\to G$, an operator $T\in\L(E,F)$ belongs to $\I(E,F)$ whenever $\iota\circ T\in\I(E,G)$. A normed operator ideal $[\I,\left\|\cdot\right\|_{\I}]$ is called injective if, in addition, $\left\|T\right\|_{\I}=\left\|\iota\circ T\right\|_{\I}$. 

The adaptation of this notion to the weighted holomorphic setting could be as follows.

A normed weighted holomorphic ideal $[\I^{\Hv^\infty},\left\|\cdot\right\|_{\I^{\Hv^\infty}}]$ is called:
\begin{itemize}
\item[(I)] injective if for any map $f\in\Hv^{\infty}(U,F)$, any complex Banach space $G$ and any into linear isometry $\iota\colon F\to G$, one has $f\in\I^{\Hv^{\infty}}(U,F)$ with $\left\|f\right\|_{\Hv^{\infty}}=\left\|\iota\circ f\right\|_{\Hv^\infty}$ whenever $\iota\circ f\in\I^{\Hv^\infty}(U,G)$. 
\end{itemize}

Given normed weighted holomorphic ideals $[\I^{\Hv^\infty},\left\|\cdot\right\|_{\I^{\Hv^\infty}}]$ and $[\J^{\Hv^\infty},\left\|\cdot\right\|_{\J^{\Hv^\infty}}]$, the relation 
$$
[\I^{\Hv^\infty},\left\|\cdot\right\|_{\I^{\Hv^\infty}}]\leq [\J^{\Hv^\infty},\left\|\cdot\right\|_{\J^{\Hv^\infty}}]
$$
means that for any complex Banach space $E$, any open set $U\subseteq E$ and any complex Banach space $F$, one has $\I^{\Hv^\infty}(U,F)\subseteq\J^{\Hv^\infty}(U,F)$ with $\left\|f\right\|_{\J^{\Hv^\infty}}\leq \left\|f\right\|_{\I^{\Hv^\infty}}$ for $f\in\I^{\Hv^\infty}(U,F)$. 

Motivated by the linear and polynomial versions (see \cite[Proposition 19.2.2]{Jar-81} and \cite[Proposition 2.3]{BotTor-20}), we next address the existence of the injective hull of a normed weighted holomorphic ideal.

Recall that the unique smallest injective operator ideal $\I^{inj}$ that contains an operator ideal $\I$ is called the injective hull of $\I$ and described as the set
$$
\I^{inj}(E,F)=\left\{T\in\L(E,F)\colon \iota_F\circ T\in\I(E,\ell_\infty(B_{Y^*})\right\},
$$
where $\iota_F\colon F\to\ell_\infty(B_{F^*})$ is the canonical isometric linear embedding defined by  
$$
\left\langle \iota_F(y),y^*\right\rangle=y^*(y)\qquad (y^*\in B_{F^*},\; y\in F).
$$
Taking $\left\|T\right\|_{\I^{inj}}=\left\|\iota_F\circ T\right\|_\I$ for $T\in\I^{inj}(E,F)$, $[\I^{inj},\left\|\cdot\right\|_{\I^{inj}}]$ is a normed (Banach) operator ideal whenever $[\I,\left\|\cdot\right\|_{\I}]$ is so. 

We now present the closely related concept in the setting of weighted holomorphic maps.

\begin{proposition}\label{p1}
Let $[\I^{\Hv^\infty},\|\cdot\|_{\I^{\Hv^\infty}}]$ be a normed (Banach) weighted holomorphic ideal. Then there exists a unique smallest normed (Banach) injective weighted holomorphic ideal $[(\I^{\Hv^\infty})^{inj},\left\|\cdot\right\|_{(\I^{\Hv^\infty})^{inj}}]$ such that
$$
[\I^{\Hv^\infty},\left\|\cdot\right\|_{\I^{\Hv^\infty}}]\leq [(\I^{\Hv^\infty})^{inj},\left\|\cdot\right\|_{(\I^{\Hv^\infty})^{inj}}].
$$
In fact, for any complex Banach space $F$, we have 
$$
(\I^{\Hv^\infty})^{inj}(U,F)=\left\{f\in\Hv^\infty(U,F)\colon \iota_F\circ f\in\I^{\Hv^\infty}(U,\ell_\infty(B_{F^*})\right\}
$$
where $\iota_F\colon F\to\ell_\infty(B_{F^*})$ is the canonical isometric linear embedding, and
$$
\left\|f\right\|_{(\I^{\Hv^\infty})^{inj}}=\left\|\iota_F\circ f\right\|_{\I^{\Hv^\infty}}\qquad (f\in(\I^{\Hv^\infty})^{inj}(U,F)).
$$
The normed (Banach) ideal $[(\I^{\Hv^\infty})^{inj},\left\|\cdot\right\|_{(\I^{\Hv^\infty})^{inj}}]$) is called the injective hull of $[\I^{\Hv^\infty},\|\cdot\|_{\I^{\Hv^\infty}}]$.
\end{proposition}

\begin{proof}
Defining the set $(\I^{\Hv^\infty})^{inj}(U,F)$ and the function $\left\|\cdot\right\|_{(\I^{\Hv^\infty})^{inj}}\colon (\I^{\Hv^\infty})^{inj}(U,F)\to\mathbb{R}^+_0$ as above, we first show that $[(\I^{\Hv^\infty})^{inj},\left\|\cdot\right\|_{(\I^{\Hv^\infty})^{inj}}]$ is an injective normed (Banach) weighted holomorphic ideal.\\

(P1) Given $f\in(\I^{\Hv^\infty})^{inj}(U,F)$, for all $x\in U$, we have
$$
v(x)\left\|f(x)\right\|=v(x)\left\|\iota_F(f(x))\right\|\leq \left\|\iota_F\circ f\right\|_v\leq \left\|\iota_F\circ f\right\|_{\I^{\Hv^\infty}}=\left\|f\right\|_{(\I^{\Hv^\infty})^{inj}},
$$
and thus $\left\|f\right\|_v\leq\left\|f\right\|_{(\I^{\Hv^\infty})^{inj}}$. Hence $f=0$ whenever $\left\|f\right\|_{(\I^{\Hv^\infty})^{inj}}=0$. It is readily to prove that $(\I^{\Hv^\infty})^{inj}(U,F)$ is a linear subspace of $\Hv^\infty(U,F)$ on which $\left\|\cdot\right\|_{\Hv^\infty}$ is absolutely homogeneous and satisfies the triangle inequality. \\

(P2) Given $h\in\Hv^\infty(U)$ and $y\in F$, we have $\iota_F\circ(h \cdot y)=h\cdot\iota_F(y)\in\I^{\Hv^\infty}(U,\ell_\infty(B_{F^*}))$ and therefore $h\cdot y\in(\I^{\Hv^\infty})^{inj}(U,F)$ with 
$\left\|h\cdot y\right\|_{(\I^{\Hv^\infty})^{inj}}=\left\|\iota_F\circ(h \cdot y)\right\|_{\I^{\Hv^\infty}}=\left\|h\cdot\iota_F(y)\right\|_{\I^{\Hv^\infty}}=\left\|h\right\|_v\left\|\iota_F(y)\right\|=\left\|h\right\|_v\left\|y\right\|$.\\

(P3) Let $V\subseteq E$ be an open set such that $V\subseteq U$, $h\in\H(V,U)$ with $c_v(h):=\sup_{x\in V}(v(x)/v(h(x)))<\infty$, $f\in\I^{\Hv^\infty}(U,F)$ and $T\in\L(F,G)$ where $G$ is a complex Banach space. Clearly, $\iota_F\circ f\in\I^{\Hv^\infty}(U,\ell_\infty(B_{F^*}))$. Since $\iota_F$ is an into linear isometry, there exists $S\in\L(\ell_\infty(B_{F^*}),\ell_\infty(B_{G^*}))$ such that $S\circ\iota_F=\iota_G\circ T$ and $\left\|S\right\|=\left\|\iota_F\circ T\right\|$ by the metric extension property of $\ell_\infty(B_{F^*})$ (see, for example, \cite[Proposition C.3.2.1]{Pie-80}). 
From $\iota_G\circ (T\circ f\circ h)=S\circ(\iota_F\circ f)\circ h\in\I^{\Hv^\infty}(U,\ell_\infty(B_{G^*}))$, we infer that $T\circ f\circ h\in(\I^{\Hv^\infty})^{inj}(U,G)$ with
\begin{align*}
\left\|T\circ f\circ h\right\|_{(\I^{\Hv^\infty})^{inj}}
&=\left\|\iota_G\circ T\circ f\circ h\right\|_{\I^{\Hv^\infty}}
=\left\|S\circ \iota_F\circ f\circ h\right\|_{\I^{\Hv^\infty}}\\
&\leq\left\|S\right\|\left\|\iota_F\circ f\right\|_{\I^{\Hv^\infty}}
=\left\|\iota_G\circ T\right\|\left\|f\right\|_{(\I^{\Hv^\infty})^{inj}}
\leq\left\|T\right\|\left\|f\right\|_{(\I^{\Hv^\infty})^{inj}}
\end{align*}

(I) Let $f\in\Hv^\infty(U,F)$ so that $\iota\circ f\in(\I^{\Hv^\infty})^{inj}(U,G)$ for any into linear isometry $\iota\colon F\to G$. The metric extension property of $\ell_\infty(B_{F^*})$ provides a $P\in\L(\ell_\infty(B_{G^*}),\ell_\infty(B_{F^*}))$ so that $P\circ\iota_G\circ\iota=\iota_F$ and $\left\|P\right\|=\left\|\iota_F\right\|=1$. 
The conditions $\iota_G\circ\iota\circ f\in\I^{\Hv^\infty}(U,\ell_\infty(B_{G^*}))$ and $\iota_F\circ f=P\circ\iota_G\circ\iota\circ f$ imply $\iota_F\circ f\in\I^{\Hv^\infty}(U,\ell_\infty(B_{F^*}))$, and so $f\in(\I^{\Hv^\infty})^{inj}(U,F)$ with  
$$
\left\|f\right\|_{(\I^{\Hv^\infty})^{inj}}
=\left\|\iota_F\circ f\right\|_{\I^{\Hv^\infty}}
=\left\|P\circ\iota_G\circ\iota\circ f\right\|_{\I^{\Hv^\infty}}
\leq\left\|P\right\|\left\|\iota_G\circ\iota\circ f\right\|_{\I^{{\Hv^\infty}}}
=\left\|\iota\circ f\right\|_{(\I^{\Hv^\infty})^{inj}}\leq\left\|f\right\|_{(\I^{\Hv^\infty})^{inj}}.
$$

On a hand, the ideal property of $[\I^{\Hv^\infty},\left\|\cdot\right\|_{\I^{\Hv^\infty}}]$ yields $[\I^{\Hv^\infty},\left\|\cdot\right\|_{\I^{\Hv^\infty}}]\leq [(\I^{\Hv^\infty})^{inj},\left\|\cdot\right\|_{(\I^{\Hv^\infty})^{inj}}]$. On the other hand, suppose $[\J^{\Hv^\infty},\left\|\cdot\right\|_{\J^{\Hv^\infty}}]$ is an injective normed weighted holomorphic ideal so that $[\I^{\Hv^\infty},\left\|\cdot\right\|_{\I^{\Hv^\infty}}]\leq[\J^{\Hv^\infty},\left\|\cdot\right\|_{\J^{\Hv^\infty}}]$. If $f\in(\I^{\Hv^\infty})^{inj}(U,F)$, one has that $\iota_F\circ f\in\I^{\Hv^\infty}(U,\ell_\infty(B_{F^*}))\subseteq\J^{\Hv^\infty}(U,\ell_\infty(B_{F^*}))$, hence $f\in\J^{\Hv^\infty}(U,F)$ with $\left\|f\right\|_{\J^{\Hv^\infty}}=\left\|\iota_F\circ f\right\|_{\J^{\Hv^\infty}}$ by the injectivity of $\J^{\Hv^\infty}$, and so $\left\|f\right\|_{\J^{\Hv^\infty}}=\left\|\iota_F\circ f\right\|_{\J^{\Hv^\infty}}\leq \left\|\iota_F\circ f\right\|_{\I^{\Hv^\infty}}=\left\|f\right\|_{(\I^{\Hv^\infty})^{inj}}$. 

The uniqueness of $[(\I^{\Hv^\infty})^{inj},\left\|\cdot\right\|_{(\I^{\Hv^\infty})^{inj}}]$ follows easily and this completes the proof.
\end{proof}

Based on the linear and polynomial variants in \cite[Proposition 19.2.2]{Jar-81} and \cite[Corollary 2.4]{BotTor-20}, respectively, the injectivity of a normed weighted holomorphic ideal is characterized by the coincidence with its injective hull. 

\begin{corollary}\label{c1}
Let $[\I^{\Hv^\infty},\|\cdot\|_{\I^{\Hv^\infty}}]$ be a normed weighted holomorphic ideal. The following are equivalent:
\begin{enumerate}
\item  $[\I^{\Hv^\infty},\|\cdot\|_{\I^{\Hv^\infty}}]$ is injective.
\item $[\I^{\Hv^\infty},\|\cdot\|_{\I^{\Hv^\infty}}]=[(\I^{\Hv^\infty})^{inj},\left\|\cdot\right\|_{(\I^{\Hv^\infty})^{inj}}]$. 
\end{enumerate}
$\hfill\qed$
\end{corollary}

Influenced by the hull procedure for the family of normed operator ideals -- introduced by Pietsch in \cite[Section 8.1]{Pie-80} --, we obtain that the correspondence $\I^{\Hv^\infty}\mapsto(\I^{\Hv^\infty})^{inj}$ is a hull procedure in the weighted holomorphic setting. 

\begin{proposition}\label{p2}
If $[\I^{\Hv^\infty},\|\cdot\|_{\I^{\Hv^\infty}}]$ and $[\J^{\Hv^\infty},\|\cdot\|_{\J^{\Hv^\infty}}]$ are normed (Banach) weighted holomorphic ideals, then: 
\begin{enumerate}
	\item $[(\I^{\Hv^\infty})^{inj},\|\cdot\|_{(\I^{\Hv^\infty})^{inj}}]$ is a normed (Banach) weighted holomorphic ideal,
	\item $[(\I^{\Hv^\infty})^{inj},\|\cdot\|_{(\I^{\Hv^\infty})^{inj}}]\leq [(\J^{\Hv^\infty})^{inj},\|\cdot\|_{(\J^{\Hv^\infty})^{inj}}]$ whenever $[\I^{\Hv^\infty},\|\cdot\|_{\I^{\Hv^\infty}}]\leq [\J^{\Hv^\infty},\|\cdot\|_{\J^{\Hv^\infty}}]$,
	\item $[((\I^{\Hv^\infty})^{inj})^{inj},\|\cdot\|_{((\I^{\Hv^\infty})^{inj})^{inj}}]=[(\I^{\Hv^\infty})^{inj},\|\cdot\|_{(\I^{\Hv^\infty})^{inj}}]$,
	\item $[\I^{\Hv^\infty},\|\cdot\|_{\I^{\Hv^\infty}}]\leq [(\I^{\Hv^\infty})^{inj},\|\cdot\|_{(\I^{\Hv^\infty})^{inj}}]$.
\end{enumerate}
\end{proposition}

\begin{proof}
$(i)$ and $(iv)$ are deduced from Proposition \ref{p1}, $(iii)$ from Corollary \ref{c1}, and $(ii)$ follows as in the last part of the proof of Proposition \ref{p1}. 
\end{proof}


\subsection{The domination property}\label{subsection 2}

The injective hull of a normed weighted holomorphic ideal can be characterized by the following domination property. This result is based on both the linear and polynomial versions stated respectively in \cite[Lemma 3.1]{BotCamSan-17} and \cite[Theorem 3.4]{BotTor-20}.

\begin{theorem}\label{t1}
Let $[\I^{\Hv^\infty},\left\|\cdot\right\|_{\I^{\Hv^\infty}}]$ be a normed weighted holomorphic ideal, let $F$ be a complex Banach space and let $f\in\Hv^\infty(U,F)$. The following assertions are equivalent:
\begin{enumerate}
\item $f$ belongs to $(\I^{\Hv^\infty})^{inj}(U,F)$.
\item There exists a complex normed space $G$ and a mapping $g\in\I^{\Hv^\infty}(U,G)$ such that
$$
\left\|\sum_{i=1}^n\lambda_i v(x_i)f(x_i)\right\|\leq\left\|\sum_{i=1}^n\lambda_iv(x_i)g(x_i)\right\|
$$
for all $n\in\N$, $\lambda_1,\ldots,\lambda_n\in\C$ and $x_1,\ldots,x_n\in U$. 
\end{enumerate}
In this case, $\left\|f\right\|_{(\I^{\Hv^\infty})^{inj}}=\inf\left\{\left\|g\right\|_{\I^{\Hv^\infty}}\right\}$, where the infimum is taken over all spaces $G$ and all mappings $g\in\I^{\Hv^\infty}(U,G)$ as in $(ii)$, and this infimum is attained.
\end{theorem}

\begin{proof}
$(i)\Rightarrow(ii)$: Suppose  that $f\in(\I^{\Hv^\infty})^{inj}(U,F)$. Take $G=\ell_\infty(B_{F^*})$ and $g=\iota_F\circ f$. Clearly, $g\in\I^{\Hv^\infty}(U,G)$ with $\left\|g\right\|_{\I^{\Hv^\infty}}=\left\|\iota_F\circ f\right\|_{\I^{\Hv^\infty}}=\left\|f\right\|_{(\I^{\Hv^\infty})^{inj}}$. Set $n\in\N$, $\lambda_1,\ldots,\lambda_n\in\C$ and $ x_1,\ldots, x_n\in U$.  
We can take some $y^*\in B_{E^*}$ so that $\left|y^*\left(\sum_{i=1}^n\lambda_iv(x_i)f( x_i)\right)\right|=\left\|\sum_{i=1}^n\lambda_i v(x_i)f(x_i)\right\|$, and thus 
\begin{align*}
\left\|\sum_{i=1}^n\lambda_i v(x_i)f( x_i)\right\|
&=\left|\sum_{i=1}^n\lambda_iv(x_i)y^*( f( x_i))\right|=\left|\sum_{i=1}^n\lambda_iv(x_i)\left\langle \iota_F( f( x_i)),y^*\right\rangle\right|\\
&=\left|\sum_{i=1}^n\lambda_iv(x_i)\left\langle  g( x_i),y^*\right\rangle\right|\leq\sup_{x^*\in B_{E^*}}\left|\sum_{i=1}^n\lambda_iv(x_i)\left\langle  g( x_i),x^*\right\rangle\right|\\
&=\sup_{x^*\in B_{E^*}}\left|\left\langle\sum_{i=1}^n\lambda_iv(x_i)g(x_i),x^*\right\rangle\right|=\left\|\sum_{i=1}^n\lambda_iv(x_i)g(x_i)\right\|
\end{align*}

$(ii)\Rightarrow(i)$: Let $G$ and $g$ be as in $(ii)$. Take $G_0=\lin( g(U))\subseteq G$ and $T_0\colon G_0\to F$ given by
$$
T_0\left(\sum_{i=1}^n\lambda_i v(x_i)g( x_i)\right)=\sum_{i=1}^n\lambda_i v(x_i)f( x_i)
$$
for all $n\in\N$, $\lambda_1,\ldots,\lambda_n\in\C$ and $ x_1,\ldots, x_n\in U$. Note that $T_0$ is well defined since   
\begin{align*}
\sum_{i=1}^n\lambda_i v(x_i)g( x_i)=\sum_{j=1}^m \alpha_j v(x_j)g( x_j)&\Rightarrow\left|\sum_{i=1}^n\lambda_i v(x_i)g( x_i)-\sum_{j=1}^m\alpha_j v(x_j)g( x_j)\right|=0\\
&\Rightarrow\left|\sum_{i=1}^n\lambda_i v(x_i)f( x_i)-\sum_{j=1}^m\alpha_j v(x_j)f( x_j)\right|=0\\
&\Rightarrow \sum_{i=1}^n\lambda_iv(x_i) f( x_i)=\sum_{j=1}^m\alpha_j v(x_j)f( x_j),
\end{align*}
by using the inequality in $(ii)$. The linearity of $T_0$ is clear, and since 
$$
\left\|T_0\left(\sum_{i=1}^n\lambda_iv(x_i) g( x_i)\right)\right\|=\left\|\sum_{i=1}^n\lambda_iv(x_i) f( x_i)\right\|\leq\left\|\sum_{i=1}^n\lambda_iv(x_i) g( x_i)\right\|
$$
for all $n\in\N$, $\lambda_1,\ldots,\lambda_n\in\C$ and $ x_1,\ldots, x_n\in U$, we deduce that $T_0$ is continuous with $\left\|T_0\right\|\leq 1$. There exists a unique operator $T\in\L(\overline{G_0},F)$ such that $T|_{G_0}=T_0$ and $\left\|T\right\|=\left\|T_0\right\|$. If $\iota\colon \overline{G_0}\to G$ is the inclusion operator, the metric extension property of $\ell_\infty(B_{G^*})$ yields an operator $S\in\L(G,\ell_\infty(B_{F^*}))$ so that $\iota_F\circ T=S\circ\iota$ and $\left\|S\right\|=\left\|\iota_F\circ T\right\|$. 
Since $(T\circ  g)(x)=T(g(x))=T_0( g(x))= f(x)$ for all $x\in U$, we have $T\circ  g= f$, and thus $\iota_F\circ  f=\iota_F\circ T\circ  g=S\circ\iota\circ  g=S\circ  g$. 
Since $g\in\I^{\Hv^\infty}(U,G)$, the ideal property of $\I^{\Hv^\infty}$ shows that $\iota_F\circ f\in\I^{\Hv^\infty}( U,\ell_\infty(B_{F^*}))$, that is, $f\in(\I^{\Hv^\infty})^{inj}(U,F)$ with $\left\|f\right\|_{(\I^{\Hv^\infty})^{inj}}=\left\|\iota_F\circ f\right\|_{\I^{\Hv^\infty}}\leq \left\|S\right\|\left\|g\right\|_{\I^{\Hv^\infty}}\leq\left\|g\right\|_{\I^{\Hv^\infty}}$. Passing to the infimum over all $G'$s and $g'$s as in $(ii)$, we conclude that $\left\|f\right\|_{(\I^{\Hv^\infty})^{inj}}\leq\inf\{\left\|g\right\|_{\I^{\Hv^\infty}}\}$.
\end{proof}

The combination of Corollary \ref{c1} and Theorem \ref{t1} immediately provides the next characterization of the injectivity of a normed weighted holomorphic ideal, that can be compared with its linear version \cite[Lemma 3.1]{BotCamSan-17} and its polynomial version \cite[Theorem 3.4]{BotTor-20}. 

\begin{corollary}\label{c3}
Let $[\I^{\Hv^\infty},\left\|\cdot\right\|_{\I^{\Hv^\infty}}]$ be a normed weighted holomorphic ideal. Then $[\I^{\Hv^\infty},\left\|\cdot\right\|_{\I^{\Hv^\infty}}]$ is injective if, and only if, given complex Banach spaces $F,G$ and mappings $f\in\Hv^\infty(U,F)$, $g\in\I^{\Hv^\infty}(U,G)$ such that 
$$
\left\|\sum_{i=1}^n\lambda_iv(x_i) f( x_i)\right\|\leq \left\|\sum_{i=1}^n\lambda_iv(x_i) g( x_i)\right\|
$$
for all $n\in\N$, $\lambda_1,\ldots,\lambda_n\in\C$ and $ x_1,\ldots, x_n\in U$, then $f\in\I^{\Hv^\infty}(U,F)$ and $\left\|f\right\|_{\I^{\Hv^\infty}}=\inf\left\{\left\|g\right\|_{\I^{\Hv^\infty}}\right\}$, where the infimum is taken over all complex Banach spaces $G$ and all such mappings $g$. $\hfill\Box$
\end{corollary}

\subsection{The injective hull of composition ideals of weighted holomorphic mappings}\label{subsection 3}

According to \cite[Definition 2.5]{CabJimKet-24}, given a normed operator ideal $[\I,\left\|\cdot\right\|_\I]$, a map $f\in\H(U,F)$ belongs to the composition ideal $\I\circ\Hv^\infty(U,F)$ if there exist a complex Banach space $G$, an operator $T\in\I(G,F)$ and a map $g\in\Hv^\infty(U,G)$ such that $f=T\circ g$. For any $f\in\I\circ\Hv^\infty(U,F)$, define 
$$
\left\|f\right\|_{\I\circ\Hv^\infty}=\inf\left\{\left\|T\right\|_\I \left\|g\right\|_v\right\},
$$
where the infimum is extended over all such factorizations of $f$. By \cite[Proposition 2.6]{CabJimKet-24}, $[\I\circ\Hv^\infty,\left\|\cdot\right\|_{\I\circ\Hv^\infty}]$ is a normed weighted holomorphic ideal.
 
We now describe the injective hull of this ideal $[\I\circ\Hv^\infty,\left\|\cdot\right\|_{\I\circ\Hv^\infty}]$. Our approach requires some preliminaries about the linearization of weighted holomorphic maps. 

Following \cite{BieSum-93,GupBaw-16}, $\G_v^\infty(U)$ is the space of all linear functionals on $\H_v^\infty(U)$ whose restriction to $B_{\H_v^\infty(U)}$ is continuous for the compact-open topology. The following result collects the properties of $\G_v^\infty(U)$ that we will need later.

\begin{theorem}\label{t0}\cite{BieSum-93,BonFri-02,GupBaw-16,Muj-91}
Let $U$ be an open set of a complex Banach space $E$ and $v$ be a weight on $U$.
\begin{enumerate}
\item  $\G_v^\infty(U)$ is a closed subspace of $\H_v^\infty(U)^*$, and the evaluation mapping $J_v\colon\H_v^\infty(U)\to\G_v^\infty(U)^*$, given by $J_v(f)(\phi)=\phi(f)$ for $\phi\in\G_v^\infty(U)$ and $f\in\H_v^\infty(U)$, is an isometric isomorphism.
\item For each $x\in U$, the evaluation functional $\delta_x\colon\H^\infty_v(U)\to\mathbb{C}$, defined by $\delta_x(f)=f(x)$ for $f\in\H^\infty_v(U)$, is in $\G^\infty_v(U)$. \item The mapping $\Delta_v\colon U\to\G_v^\infty(U)$ given by $\Delta_v(x)=\delta_x$ is in $\H_v^\infty(U,\G_v^\infty(U))$ with $\left\|\Delta_v\right\|_v\leq 1$.
\item $B_{\G^\infty_v(U)}=\overline{\abco}(\At_{\G^\infty_v(U)})\subseteq\H^\infty_v(U)^*$ and $\G_v^\infty(U)=\overline{\lin}(\At_{\G^\infty_v(U)})\subseteq\H_v^\infty(U)^*$, where $\At_{\G^\infty_v(U)}=\{v(x)\delta_x\colon x\in U\}$.
\item For each $\phi\in\lin(\At_{\G^\infty_v(U)})$, we have 
$$
\left\|\phi\right\|=\inf\left\{\sum_{i=1}^n\left|\lambda_i\right|\colon \phi=\sum_{i=1}^n\lambda_iv(x_i)\delta_{x_i}\right\}.
$$
\item For every complex Banach space $F$ and every mapping $f\in\H_v^\infty(U,F)$, there exists a unique operator $T_f\in\L(\G_v^\infty(U),F)$ such that $T_f\circ\Delta_v=f$. Furthermore, $\left\|T_f\right\|=\left\|f\right\|_v$. 
\item  For each $f\in\H^\infty_v(U,F)$, the mapping $f^t\colon F^*\to\H^\infty_v(U)$, defined by $f^t(y^*)=y^*\circ f$ for all $y^*\in F^*$, is in $\L(F^*,\H^\infty_v(U))$ with $||f^t||=\left\|f\right\|_v$ and $f^t=J_v^{-1}\circ(T_f)^*$, where $(T_f)^*\colon F^*\to\G^\infty_v(U)^*$ is the adjoint operator of $T_f$.
$\hfill\qed$
\end{enumerate}
\end{theorem}

For $v=1_U$ where $1_U(x)=1$ for all $x\in U$, it is usual to write $\H^\infty(U,F)$ (the Banach space of all bounded holomorphic mappings from $U$ into $F$, under the supremum norm) instead of $\H^\infty_v(U,F)$, $\H^\infty(U)$ rather than $\H^\infty(U,\C)$ and, following Mujica's notation in \cite{Muj-91}, $\G^\infty(U)$ instead of $\G^\infty_v(U)$.

\begin{proposition}\label{p3}
Let $[\I,\left\|\cdot\right\|_\I]$ be an operator ideal. Then 
$$
[(\I\circ\Hv^\infty)^{inj},\|\cdot\|_{(\I\circ\Hv^\infty)^{inj}}]
=[\I^{inj}\circ\Hv^\infty,\|\cdot\|_{\I^{inj}\circ\Hv^\infty}].
$$
In particular, the weighted holomorphic ideal $[\I^{inj}\circ\Hv^\infty,\|\cdot\|_{\I^{inj}\circ\Hv^\infty}]$ is injective.
\end{proposition}

\begin{proof}
Let $F$ be a complex Banach space and $f\in(\I\circ\Hv^\infty)^{inj}(U,F)$. Hence $\iota_F\circ f\in\I\circ\Hv^\infty(U,\ell_\infty(B_{F^*}))$, and so $\iota_F\circ f=T\circ g$ for some complex Banach space $G$, an operator $T\in\I(G,\ell_\infty(B_{F^*}))$ and a map $g\in\Hv^\infty(U,G)$. By Theorem \ref{t0}, we can find two operators $T_f\in\L(\G^\infty_v(U),F)$ and $T_g\in\L(\G^\infty_v(U),G)$ with $\|T_f\| = \|f\|_v$ and $\|T_g\|=\|g\|$ such that $T_f\circ\Delta_v= f$ and $T_g\circ\Delta_v= g$. Since $\G^\infty_v(U)=\overline{\lin}(\Delta_v(U))\subseteq\Hv^\infty(U)^*$ and 
$$
\iota_F\circ T_f\circ\Delta_v=\iota_F\circ  f=T\circ  g=T\circ T_g\circ\Delta_v,
$$
it follows that $\iota_F\circ T_f=T\circ T_g$, and thus $\iota_F\circ T_f\in\I(\G^\infty_v(U),\ell_\infty(B_{F^*}))$, that is, $T_f\in \I^{inj}(\G^\infty_v(U),F)$. Hence $ f=T_f\circ\Delta_v\in\I^{inj}\circ\Hv^\infty(U,F)$. Moreover, 
$$
\left\|f\right\|_{\I^{inj}\circ\Hv^\infty}\leq\left\|T_f\right\|_{\I^{inj}}\|\Delta_v\|\leq \left\|T_f\right\|_{\I^{inj}}=\left\|\iota_F\circ T_f\right\|_{\I}\leq\left\|T\right\|_\I\left\|T_g\right\|=\left\|T\right\|_\I\|g\|_v,
$$
and passing to the infimum over all the factorizations of $\iota_F\circ f$ yields $\left\|f\right\|_{\I^{inj}\circ\Hv^\infty}\leq \left\|\iota_F\circ f\right\|_{\I\circ\Hv^\infty}=\left\|f\right\|_{(\I\circ\Hv^\infty)^{inj}}$.

Conversely, let $f\in\I^{inj}\circ\Hv^\infty(U,F)$. Hence $f=T\circ g$ for some complex Banach space $G$, $T\in\I^{inj}(G,F)$ and $g\in\Hv^\infty(U,G)$. Therefore $\iota_F\circ f=(\iota_F\circ T)\circ g\in\I\circ\Hv^\infty(U,\ell_\infty(B_{F^*})$, and thus $f\in(\I\circ\Hv^\infty)^{inj}(U,F)$ with 
$$
\left\|f\right\|_{(\I\circ\Hv^\infty)^{inj}}=\left\|\iota_F\circ f\right\|_{\I\circ\Hv^\infty}=\left\|\iota_F\circ T\circ g\right\|_{\I\circ\Hv^\infty}
\leq\left\|\iota_F\circ T\right\|_{\I}\|g\|_v=\left\|T\right\|_{\I^{inj}}\|g\|_v.
$$
Taking the infimum over all the factorizations of $f$, we conclude that $\left\|f\right\|_{(\I\circ\Hv^\infty)^{inj}}\leq \left\|f\right\|_{\I^{inj}\circ\Hv^\infty}$.
\end{proof}

From Proposition \ref{p3} and Corollary \ref{c1}, we deduce the following.

\begin{corollary}\label{c4}
Let $[\I,\left\|\cdot\right\|_\I]$ be an injective normed operator ideal. Then $[\I\circ\Hv^\infty,\left\|\cdot\right\|_{\I\circ\Hv^\infty}]$ is an injective weighted holomorphic ideal.$\hfill$ $\qed$
\end{corollary}

For $\I=\F,\overline{\F},\K,\W,\S,\R,\AS$ and Banach spaces $E,F$, we will denote by $\I(E,F)$ the linear space of all finite-rank (approximable, compact, weakly compact, separable, Rosenthal, Asplund) bounded linear operators from $E$ to $F$, respectively. The components $\I(E,F)$, equipped with the operator canonical norm $\left\|\cdot\right\|$, generate a normed operator ideal (see \cite{Pie-80}).

For a map $f\in\H(U,F)$, the v-range of $f$ is the set  
$$
(vf)(U)=\left\{v(x)f(x)\colon x\in U\right\}\subseteq F.
$$
Note that $f$ belongs to $\Hv^\infty(U,F)$ if and only if $(vf)(U)$ is a norm-bounded subset of $F$. This motivates the following concepts.

\begin{definition}
Let $U$ be an open set of a complex Banach space $E$, let $v$ be a weight on $U$ and let $F$ be a complex Banach space. 

A mapping $f\in\Hv^\infty(U,F)$ is said to be v-compact (resp., v-weakly compact, v-separable, v-Rosenthal, v-Asplund) if $(vf)(U)$ is a relatively compact (resp., relatively weakly compact, separable, Rosenthal, Asplund) subset of $F$. 

A mapping $f\in\Hv^\infty(U,F)$ is said to have finite dimensional v-rank if  $(vf)(U)$ is a finite dimensional subspace of $F$, and $f$ is said to be v-approximable if it is the limit in the v-norm of a sequence of finite v-rank weighted holomorphic mappings of $\Hv^\infty(U,F)$. 

For $\I=\F,\overline{\F},\K,\W,\S,\R,\AS$, $\H_{v\I}^\infty(U,F)$ stand for the linear space of all finite v-rank (resp., v-approximable, v-compact, v-weakly compact, v-separable, v-Rosenthal, v-Asplund) weighted holomorphic mappings from $U$ into $F$.
\end{definition}

The same proofs of Theorem 2.9 and Corollary 2.10 in \cite{CabJimKet-24} yield the following two results. 

\begin{theorem}\label{linear}
Let $f\in\Hv^\infty(U,F)$ and $\I=\F,\overline{\F},\K,\W,\S,\R,\AS$. For the normed operator ideal $[\I,\left\|\cdot\right\|_\I]$, the following are equivalent:
\begin{enumerate}
\item $f$ belongs to $\H_{v\I}^\infty(U,F)$.
\item $T_f$ belongs to $\I(\Gv^\infty(U),F)$.
\end{enumerate}
In this case, $\left\|f\right\|_v=\left\|T_f\right\|_I$. Furthermore, the correspondence $f\mapsto T_f$ is an isometric isomorphism from $(\H_{v\I}^\infty(U,F),\left\|\cdot\right\|_v)$ onto $(\I(\Gv^\infty(U),F),\left\|\cdot\right\|_I)$. $\hfill\Box$
\end{theorem}

\begin{corollary}\label{proposition: Banach ideal}
$[\H_{v\I}^\infty,\left\|\cdot\right\|_v]=[\I\circ\H_v^\infty,\left\|\cdot\right\|_{\I\circ\H_v^\infty}]$ for $\I=\F,\overline{\F},\K,\W,\S,\R,\AS$. As a consequence, 
\begin{enumerate}
\item $[\H_{v\I}^\infty,\left\|\cdot\right\|_v]$ is a Banach weighted holomorphic ideal for $\I=\overline{\F},\K,\W,\S,\R,\AS$,
\item $[\H_{v\F}^\infty,\left\|\cdot\right\|_v]$ is a normed weighted holomorphic ideal. 
\end{enumerate}
$\hfill\qed$
\end{corollary}

We are in a position to establish the injectivity of these ideals.

\begin{corollary}\label{c5}
For $\I=\F,\K,\W,\S,\R,\AS$, the weighted holomorphic ideal $[\H_{v\I}^\infty,\|\cdot\|_v]$ is injective.
\end{corollary}

\begin{proof}
Applying Corollary \ref{proposition: Banach ideal} for the first and fourth equalities, Proposition \ref{p3} for the second, and \cite{GonGut-99} for the third, one has
\begin{align*}
[(\H_{v\I}^\infty)^{inj},(\|\cdot\|_v)^{inj}]
&=[(\I\circ\Hv^\infty)^{inj},\|\cdot\|_{(\I\circ\Hv^\infty)^{inj}}]
=[\I^{inj}\circ\Hv^\infty,\|\cdot\|_{\I^{inj}\circ\Hv^\infty}]\\
&=[\I\circ\Hv^\infty,\|\cdot\|_{\I\circ\Hv^\infty}]
=[\H_{v\I}^\infty,\|\cdot\|_v].
\end{align*}
\end{proof}

We now identify the injective hull of the ideal $\H_{v\overline{\F}}^\infty$.

\begin{corollary}\label{c6}
$[(\H_{v\overline{\F}}^\infty)^{inj},(\|\cdot\|_v)^{inj}]=[\H_{v\K}^\infty,\|\cdot\|_v]$.
\end{corollary}

\begin{proof}
As in the preceding proof, one now has  
\begin{align*}
[(\H_{v\overline{\F}}^\infty)^{inj},(\|\cdot\|_v)^{inj}]
&=[(\overline{\F}\circ\Hv^\infty)^{inj},\|\cdot\|_{(\overline{\F}\circ\Hv^\infty)^{inj}}]
=[(\overline{\F})^{inj}\circ\Hv^\infty,\|\cdot\|_{(\overline{\F})^{inj}\circ\Hv^\infty}]\\
&=[\K\circ\Hv^\infty,\|\cdot\|_{\K\circ\Hv^\infty}]
=[\H_{v\K}^\infty,\|\cdot\|_v]
\end{align*}
by Corollary \ref{proposition: Banach ideal} for the first and fourth equalities, Proposition \ref{p3} for the second, and the equality $[(\overline{\F})^{inj},\left\|\cdot\right\|_{inj}]=[\K,\left\|\cdot\right\|]$ by \cite[Proposition 4.6.13]{Pie-80} for the third.
\end{proof}


\subsection{The injective hull of dual ideals of weighted holomorphic mappings}\label{subsection 4}

Following \cite[Section 4.4]{Pie-80}, given a normed operator ideal $[\I,\left\|\cdot\right\|_\I]$, the components  
$$
\I^\d(E,F):=\left\{T\in\L(E,F)\colon T^*\in\I(F^*,E^*)\right\}
$$
for any normed spaces $E$ and $F$, endowed with the norm 
$$
\left\|T\right\|_{\I^\d}=\left\|T^*\right\|_\I\qquad (T\in\I^\d(E,F)),
$$ 
define a normed operator ideal, $[\I^\d,\left\|\cdot\right\|_{\I^d}]$, called dual ideal of $\I$. Moreover, $[\I,\left\|\cdot\right\|_\I]$ is said to be symmetric and completely symmetric if $[\I,\left\|\cdot\right\|_\I]\leq [\I^\d,\left\|\cdot\right\|_{\I^\d}]$ and $[\I,\left\|\cdot\right\|_\I]=[\I^\d,\left\|\cdot\right\|_{\I^\d}]$, respectively.

Based on the notion of transpose of a weighted holomorphic map (see Theorem \ref{t0}), we introduce the concept of dual weighted holomorphic ideal of an operator ideal $\I$.

\begin{definition}
Let $\I$ be an operator ideal. For any open subset $U$ of a complex Banach space $E$, any weight $v$ on $U$ and any complex Banach space $F$, we define 
$$
\I^{\Hv^\infty\text{-}\d}(U,F)= \{f \in \Hv^\infty(U,F)\colon f^t \in \I(F^*,\Hv^\infty(U))\}
$$
If $[\I,\left\|\cdot\right\|_\I]$ is a normed operator ideal, we set 
$$
\|f\|_{\I^{\Hv^\infty\text{-}\d}}=\|f^t\|_\I\qquad (f \in \I^{\Hv^\infty\text{-}\d}(U,F)).
$$
\end{definition}
 
We now show that $[\I^{\Hv^\infty\text{-}\d},\|\cdot\|_{\I^{\Hv^\infty\text{-}\d}}]$ is in fact a normed weighted holomorphic ideal.

\begin{theorem}\label{theorem: transposition}
Let $\I$ be an operator ideal. The following statements about a mapping $f\in \Hv^\infty(U,F)$ are equivalent:
\begin{enumerate}
\item $f$ belongs to $\I^{\Hv^\infty\text{-}\d}(U,F)$.
\item $f$ belongs to  $\I^\d\circ\Hv^\infty(U,F)$. 
\end{enumerate}
If in addition $(\I,\left\|\cdot\right\|_\I)$ is a normed operator ideal, then 
$$
\left\|f\right\|_{\I^{\Hv^\infty\text{-}\d}}=\left\|f\right\|_{\I^\text{dual}\circ\Hv^\infty}\qquad (f \in \I^{\Hv^\infty\text{-}\d}(U,F)).
$$
\end{theorem}

\begin{proof}
$(i)\Rightarrow (ii)$: Let $f\in\I^{\Hv^\infty\text{-}\d}(U,F)$. Then $f^t \in\I(F^*,\Hv^\infty(U))$. By Theorem \ref{t0}, we can take $T_f \in\L(\Gv^\infty(U),F)$ such that $T_f \circ \Delta_v=f$ and also $(T_f)^*= J_v\circ f^t$. Hence $(T_f)^*\in \I(F^*,\Gv^\infty(U)^*)$ and therefore $T_f \in \I^\d(\Gv^\infty,F)$. Thus, by \cite[Theorem 2.7]{CabJimKet-24} we have $f \in \I^\d \circ\Hv^\infty(U,F)$ with $\|f\|_{\I^\d \circ \Hv^\infty} = \|T_f\|_{\I^\d}$. Further,
$$
\|f\|_{\I^\d \circ\Hv^\infty} = \|T_f\|_{\I^\d} = \|(T_f)^*\|_\I= \|J_v \circ f^t\|_\I \leq \|J_v\| \|f^t\|_\I= \|f\|_{\I^{\Hv^\infty\text{-}\d}}.
$$
$(ii)\Rightarrow (i)$: Let $f\in\I^\d\circ\Hv^\infty(U,F)$. Then there are a complex Banach space $G$, a map $g\in\Hv^\infty(U,G)$ and an operator $T\in\I^\text{dual}(G,F)$ such that $f=T\circ g$. Given $y^*\in F^*$, we have
$$
f^t(y^*)=(T\circ g)^t(y^*)=y^*\circ(T\circ g)=(y^*\circ T)\circ g=T^*(y^*)\circ g=g^t(T^*(y^*))=(g^t\circ T^*)(y^*),
$$
and thus $f^t=g^t\circ T^*$. Since $T^*\in\I(F^*,G^*)$ and $g^t\in\L(G^*,\Hv^\infty(U))$, we obtain that $f^t\in\I(F^*,\Hv^\infty(U))$. Hence $f\in\I^{\Hv^\infty\text{-}\d}(U,F)$ and since   
$$
\left\|f\right\|_{\I^{\Hv^\infty\text{-}\d}}=\left\|f^t\right\|_{\I}=\left\|g^t\circ T^*\right\|_{\I}\leq\left\|g^t\right\|\left\|T^*\right\|_{\I}=\|g\|_v\left\|T\right\|_{\I^\mathrm{dual}},
$$
and taking the infimum over all representations $T\circ g$ of $f$, we conclude that $\left\|f\right\|_{\I^{\Hv^\infty\text{-}\d}}\leq\left\|f\right\|_{\I^\mathrm{dual}\circ\Hv^\infty}$.
\end{proof}

Theorem \ref{theorem: transposition} enables us to include the following description of the dual weighted holomorphic ideal of a completely symmetric normed operator ideal. 

\begin{corollary}\label{cor-dual}
$[\I^{\Hv^\infty\text{-}\d},\left\|\cdot\right\|_{\I^{\Hv^\infty\text{-}\d}}]=[\I\circ\Hv^\infty,\left\|\cdot\right\|_{\I\circ\Hv^\infty}]$ whenever $[\I,\left\|\cdot\right\|_\I]$ is a completely symmetric normed operator ideal. $\hfill \square$
\end{corollary}

The operator ideal $\I=\F,\overline{\F},\K,\W$ is completely symmetric by \cite[Proposition 4.4.7]{Pie-80}. Then Corollaries \ref{cor-dual} and \ref{proposition: Banach ideal} give us the following identifications.

\begin{corollary}
$[\I^{\Hv^\infty\text{-}\d},\left\|\cdot\right\|_{\I^{\Hv^\infty\text{-}\d}}]=[\H_{v\I}^\infty,\left\|\cdot\right\|_v]$ for $\I=\F,\overline{\F},\K,\W$. $\hfill \square$
\end{corollary}

On the injectivity property, we now can give the following.

\begin{corollary}\label{c10}
If $[\I,\left\|\cdot\right\|_\I]$ is a completely symmetric injective normed operator ideal, then the weighted holomorphic ideal $[\I^{\Hv^\infty\text{-}\d},\left\|\cdot\right\|_{\I^{\Hv^\infty\text{-}\d}}]$ is injective.
\end{corollary}

\begin{proof}
Applying Theorem \ref{theorem: transposition}, Proposition \ref{p3} and the properties of $[\I,\left\|\cdot\right\|_\I]$, we have
\begin{align*}
[(\I^{\Hv^\infty\text{-}\d})^{inj},\left\|\cdot\right\|_{(\I^{\Hv^\infty\text{-}\d})^{inj}}]
&=[(\I^{dual}\circ\Hv^\infty)^{inj},\left\|\cdot\right\|_{(\I^{dual}\circ\Hv^\infty)^{inj}}]=[(\I^{dual})^{inj}\circ\Hv^\infty,\left\|\cdot\right\|_{(\I^{dual})^{inj}\circ\Hv^\infty}]\\
&=[\I^{inj}\circ\Hv^\infty,\left\|\cdot\right\|_{\I^{inj}\circ\Hv^\infty}]=[\I\circ\Hv^\infty,\left\|\cdot\right\|_{\I\circ\Hv^\infty}]\\
&=[\I^{dual}\circ\Hv^\infty,\left\|\cdot\right\|_{\I^{dual}\circ\Hv^\infty}]=[\I^{\Hv^\infty\text{-}\d},\left\|\cdot\right\|_{\I^{\Hv^\infty\text{-}\d}}],
\end{align*}
and the result follows from Corollary \ref{c1}.
\end{proof}

Now, we describe the dual weighted holomorphic ideals of both the ideal $\K_p$ of $p$-compact operators \cite{Pie-13} and the ideal $\Du_p$ of Cohen strongly $p$-summing operators \cite{Coh-73}. As usual, $\Nu_p$ denotes the ideal of $p$-nuclear operators, $\I_p$ the ideal of $p$-integral operators, and $\Pi_p$ the ideal of absolutely $p$-summing operators (see \cite{Pie-80}).

\begin{corollary}\label{c11}
Let $\I$ and $\J$ be Banach operator ideals such that $\I^{\d}=\J^{inj}$. Then $\I^{\Hv^\infty\text{-}\d}=(\J\circ\Hv^\infty)^{inj}$. As a consequence, $\K_p^{\Hv^\infty\text{-}\d}=(\Nu_p\circ\Hv^\infty)^{inj}$ and $\Du_p^{\Hv^\infty\text{-}\d}=(\I_{p^*}\circ\Hv^\infty)^{inj}$ for any $p\in (1,\infty)$, where $p^*$ denotes the H\"older conjugate of $p$.
\end{corollary}

\begin{proof}
The combination of Theorem \ref{theorem: transposition} and Proposition \ref{p3} gives
$$
\I^{\Hv^\infty\text{-}\d}=\I^{dual}\circ\Hv^\infty=\J^{inj}\circ\Hv^\infty=(\J\circ\Hv^\infty)^{inj}.
$$
This equality yields the consequence in view that $\K_p^{\d}=\Nu_p^{inj}$ by \cite[Theorem 6]{Pie-13}, and $\Du_p^{\d}=\Pi_{p^*}=\I_{p^*}^{inj}$ by \cite{Coh-73} and \cite[Theorem 2.9.7]{Pie-13}.
\end{proof}



\subsection{The closed injective hull of ideals of weighted holomorphic mappings}\label{subsection 5}

According to \cite[Section 4.2.1]{Pie-80}, given an operator ideal $\I$ and Banach spaces $E,F$, an operator $T\in\L(E,F)$ is in the closure of $\I(E,F)$ in $(\L(E,F),\left\|\cdot\right\|)$, denoted by $\overline{\I}(E,F)$, if there exists a sequence $(T_n)$ in $\I(E,F)$ such that $\lim_{n\to\infty}\left\|T_n-T\right\|=0$. In this way, the components $\overline{\I}(E,F)$ define an operator ideal $\overline{\I}$.

This concept motivates the following in the setting of weighted holomorphic maps.

\begin{definition}
Let $U$ be an open set of a complex Banach space $E$, let $v$ be a weight on $U$ and let $F$ be a complex Banach space. Given a weighted holomorphic ideal $\I^{\Hv^\infty}$, a map $f\in\Hv^\infty(U,F)$ is said to belong to the closure of $\I^{\Hv^\infty}(U,F)$ in $(\Hv^\infty(U,F),\|\cdot\|_v)$, and it is denoted by $f\in\overline{\I^{\Hv^\infty}}(U,F)$, if there exists a sequence $(f_n)$ in $\I^{\Hv^\infty}(U,F)$ such that $\lim_{n\to\infty}\|f_n-f\|_v=0$.  
\end{definition}

It is easy to prove the following result.

\begin{proposition}
Let $\I^{\Hv^\infty}$ be a weighted holomorphic ideal. Then $\overline{\I^{\Hv^\infty}}$ is a weighted holomorphic ideal containing $\I^{\Hv^\infty}$, and it is called the closure of $\I^{\Hv^\infty}$.

We say that $\I^{\Hv^\infty}$ is closed if $\I^{\Hv^\infty}=\overline{\I^{\Hv^\infty}}$, and we call closed injective hull of $\I^{\Hv^\infty}$ to the injective hull of the ideal $\overline{\I^{\Hv^\infty}}$ and it is denoted by $(\overline{\I^{\Hv^\infty}})^{inj}$. $\hfill\qed$
\end{proposition}

The closed injective hull of a weighted holomorphic ideal of composition type admits the following description.

\begin{proposition}\label{p4}
Let $[\I,\left\|\cdot\right\|_\I]$ be an operator ideal. Then 
$$
[(\overline{\I\circ\Hv^\infty})^{inj},\|\cdot\|_{(\overline{\I\circ\Hv^\infty})^{inj}}]=[(\overline{\I})^{inj}\circ\Hv^\infty,\|\cdot\|_{(\overline{\I})^{inj}\circ\Hv^\infty}].
$$
In particular, the weighted holomorphic ideal $[(\overline{\I})^{inj}\circ\Hv^\infty,\|\cdot\|_{(\overline{\I})^{inj}\circ\Hv^\infty}]$ is injective.
\end{proposition}

\begin{proof}
We claim that $\overline{\I}\circ\Hv^\infty(U,F)=\overline{\I\circ\Hv^\infty}(U,F)$. 
Indeed, note first that $\overline{\I}\circ\Hv^\infty(U,F)$ is closed: let $f\in\Hv^\infty(U,F)$ and assume that $(f_n)_{n\in\N}$ is a sequence in $\overline{\I}\circ\Hv^\infty(U,F)$ such that $\|f_n-f\|_v\to 0$ as $n\to\infty$; since $T_{f_n}\in\overline{\I}(\Gv^\infty(U),F)$ by Theorem \ref{linear} and $\left\|T_{f_n}-T_f\right\|=\|f_n-f\|_v$ for all $n\in\N$ by Theorem \ref{t0}, we have that $T_f\in\overline{\I}(\Gv^\infty(U),F)$, and thus $f\in\overline{\I}\circ\Hv^\infty(U,F)$ again by Theorem \ref{linear}. 

Now, from $\I\circ\Hv^\infty(U,F)\subseteq\overline{\I}\circ\Hv^\infty(U,F)$, we infer that $\overline{\I\circ\Hv^\infty}(U,F)\subseteq\overline{\I}\circ\Hv^\infty(U,F)$. For the converse, take $f\in\overline{\I}\circ\Hv^\infty(U,F)$; hence $T=T\circ g$ for some complex Banach space $G$, $T\in\overline{\I}(G,F)$ and $g\in\Hv^\infty(U,G)$; thus we can find a sequence $(T_n)_{n\in\N}$ in $\I(G,F)$ such that $\left\|T_n-T\right\|\to 0$ as $n\to\infty$, and since $\|T_n\circ g-T\circ g\|_v=\|(T_n-T)\circ g\|_v\leq\left\|T_n-T\right\|\|g\|_v$ for all $n\in\N$, we deduce that $f\in\overline{\I\circ\Hv^\infty}(U,F)$, and this proves our claim.

Now, using Proposition \ref{p3}, we conclude that
$$
[(\overline{\I})^{inj}\circ\Hv^\infty,\|\cdot\|_{(\overline{\I})^{inj}\circ\Hv^\infty}]
=[(\overline{\I}\circ\Hv^\infty)^{inj},\left\|\cdot\right\|_{(\overline{\I}\circ\Hv^\infty)^{inj}}]
=[(\overline{\I\circ\Hv^\infty})^{inj},\|\cdot\|_{(\overline{\I\circ\Hv^\infty})^{inj}}].
$$
\end{proof}

In terms of an Ehrling-type inequality \cite{Ehr-54}, Jarchow and Pelczy\'nski characterized the closed injective hull of a Banach operator ideal in \cite[Theorem 20.7.3]{Jar-81}. We now present a variant of this result for weighted holomorphic maps.
 
\begin{theorem}\label{t2}
For a weighted holomorphic ideal $\I^{\Hv^\infty}$ and $f\in\Hv^\infty(U,F)$, the following are equivalent:
\begin{enumerate}
\item $f$ belongs to $(\overline{\I^{\Hv^\infty}})^{inj}(U,F)$.
\item For each $\varepsilon>0$, there are a complex normed space $G_\varepsilon$ and a mapping $g_\varepsilon\in\I^{\Hv^\infty}(U,G_\varepsilon)$ such that
$$
\left\|\sum_{i=1}^n\lambda_i v(x_i)f( x_i)\right\|\leq\left\|\sum_{i=1}^n\lambda_i v(x_i) g_\varepsilon( x_i)\right\|+\varepsilon\sum_{i=1}^n\left|\lambda_i\right|
$$
for all $n\in\N$, $\lambda_1,\ldots,\lambda_n\in\C$ and $x_1,\ldots, x_n\in U$. 
\end{enumerate}
\end{theorem}

\begin{proof} 
$(i)\Rightarrow(ii)$: Let $f\in(\overline{\I^{\Hv^\infty}})^{inj}(U,F)$ and $\varepsilon>0$. Hence $\iota_F\circ f\in\overline{\I^{\Hv^\infty}}(U,\ell_\infty(B_{F^*}))$ and so we can find a map $g_\varepsilon\in\I^{\Hv^\infty}(U,\ell_\infty(B_{F^*}))$ such that $\|\iota_F\circ f-g_\varepsilon\|_v<\varepsilon$. For any $n\in\N$, $\lambda_1,\ldots,\lambda_n\in\C$ and $ x_1,\ldots, x_n\in U$, we obtain 
\begin{align*}
\left\|\sum_{i=1}^n\lambda_iv(x_i)(\iota_F(f( x_i))-g_\varepsilon( x_i))\right\|
&\leq\sum_{i=1}^n\left|\lambda_i\right|v(x_i)\left\|(\iota_F\circ f-g_\varepsilon)(x_i)\right\|\\
&\leq\sum_{i=1}^n\left|\lambda_i\right|\|\iota_F\circ f-g_\varepsilon\|_v\leq\varepsilon\sum_{i=1}^n\left|\lambda_i\right|,
\end{align*}
and therefore 
\begin{align*}
\left\|\sum_{i=1}^n\lambda_i v(x_i)f( x_i)\right\|
&=\left\|\iota_F\left(\sum_{i=1}^n\lambda_i v(x_i)f( x_i)\right)\right\|\\
&\leq\left\|\sum_{i=1}^n\lambda_iv(x_i)g_\varepsilon( x_i)\right\|+\left\|\sum_{i=1}^n\lambda_iv(x_i)(\iota_F( f( x_i))-g_\varepsilon( x_i))\right\|\\
&\leq\left\|\sum_{i=1}^n\lambda_iv(x_i)g_\varepsilon( x_i)\right\|+\varepsilon\sum_{i=1}^n\|\lambda_i\|.
\end{align*}

$(ii)\Rightarrow(i)$: Let $\varepsilon>0$ and $\phi=\sum_{i=1}^n\lambda_iv(x_i)\delta_{ x_i}\in\lin(\At_{\G^\infty_v(U)})$. By $(ii)$, we have a complex normed space $G_\varepsilon$ and a map $g_\varepsilon\in\I^{\Hv^\infty}(U,G_\varepsilon)$ satisfying that
\begin{align*}
\left\|T_f(\phi)\right\|&
=\left\|\sum_{i=1}^n\lambda_iv(x_i)T_f(\delta_{ x_i})\right\|
=\left\|\sum_{i=1}^n\lambda_i v(x_i)f( x_i)\right\|\\
&\leq\left\|\sum_{i=1}^n\lambda_i v(x_i)g_\varepsilon( x_i)\right\|+\varepsilon\sum_{i=1}^n\left|\lambda_i\right|\\
&=\left\|\sum_{i=1}^n\lambda_iv(x_i)T_{g_\varepsilon}(\delta_{ x_i})\right\|+\varepsilon\sum_{i=1}^n\left|\lambda_i\right|\\
&=\left\|T_{g_\varepsilon}(\phi)\right\|+\varepsilon\sum_{i=1}^n\left|\lambda_i\right|,
\end{align*}
and taking the infimum over all the representations of $\phi$, Theorem \ref{t0} gives  
$$
\left\|T_f(\phi)\right\|\leq \left\|T_{g_\varepsilon}(\phi)\right\|+\varepsilon\left\|\phi\right\|.
$$
Consider the Banach space $F_\varepsilon=G_\varepsilon\oplus_{\ell_1}\Gv^\infty(U)$ and define the map $R_\varepsilon\colon\Gv^\infty(U)\to  F_\varepsilon$ by $R_\varepsilon(\phi)=(T_{g_\varepsilon}(\phi),\varepsilon\phi)$. Clearly, $R_\varepsilon$ is an injective continuous linear operator with $\left\|R_\varepsilon\right\|\leq\|g_\varepsilon\|_v+\varepsilon$. By the inequality above, the map $S_\varepsilon\colon R_\varepsilon(\Gv^\infty(U))\to F$ given by $S_\varepsilon(R_\varepsilon(\phi))=T_f(\phi)$ is well defined. Clearly, $S_\varepsilon$ is linear and since 
$$
\left\|S_\varepsilon(R_\varepsilon(\phi))\right\|=\left\|T_f(\phi)\right\|\leq \left\|T_{g_\varepsilon}(\phi)\right\|+\varepsilon\left\|\phi\right\|=\left\|(T_{g_\varepsilon}(\phi),\varepsilon\phi)\right\|=\left\|R_\varepsilon(\phi)\right\|
$$
for all $\phi\in\Gv^\infty(U)$, it is continuous with $\left\|S_\varepsilon\right\|\leq 1$. By the metric extension property of $\ell_\infty(B_{F^*})$, there exists an operator $T_\varepsilon\in\L(F_\varepsilon,\ell_\infty(B_{F^*}))$ such that $\iota_F\circ S_\varepsilon=T_\varepsilon|_{R_\varepsilon(\G^\infty_v(U))}$ and $\left\|T_\varepsilon\right\|=\left\|S_\varepsilon\right\|$. 
$$
\begin{tikzpicture}
  \node (D) {$U$};
	\node (GD) [right of=D] {$\G^\infty_v(U)$};
	\node (RG) [above right of=GD] {$R_\varepsilon(\G^\infty_v(U))$};
  \node (X)  [right of=GD] {$F$};
  \node (LI) [right of=X] {$\ell_\infty(B_{F^*})$};
	\node (Y) [below of=D] {$G_\varepsilon$};
  \node (Z) [right of=Y] {$ F_\varepsilon$};
	\draw[->] (D)  to node {$\Delta_v$} (GD);
  \draw[->] (GD) to node {$R_\varepsilon$} (RG);
	\draw[->] (GD) to node {$T_f$} (X);
	\draw[->] (D)  to node [swap] {$g_\varepsilon$} (Y);
  \draw[->] (RG) to node {$S_\varepsilon$} (X);
	\draw[->] (X)  to node {$\iota_F$} (LI);
	\draw[->] (Y)  to node [swap] {$p_1$} (Z);
  \draw[->] (Z)  to node [swap] {$T_\varepsilon$} (LI);
	\draw[->] (GD) to node {$R_\varepsilon$} (Z);
  \draw[->] (GD) to node [swap] {$p_2$} (Z);
\end{tikzpicture}
$$
Define now the maps $h_\varepsilon,k_\varepsilon\colon U\to\ell_\infty(B_{F^*})$ by $h_\varepsilon(x)=T_\varepsilon(g_\varepsilon(x),0)$ and $k_\varepsilon(x)=T_\varepsilon(0,\varepsilon\Delta_v(x))$ for all $x\in U$. On a hand, $h_\varepsilon=T_\varepsilon\circ p_1\circ g_\varepsilon\in\I^{\Hv^\infty}(U,\ell_\infty(B_{F^*}))$, where $p_1\colon G_\varepsilon\to  F_\varepsilon$ is the linear continuous map defined by $p_1(y)=(y,0)$, and, on the other hand, $k_\varepsilon=T_\varepsilon\circ p_2\circ \varepsilon\Delta_v\in\Hv^\infty(U,\ell_\infty(B_{F^*}))$, where $p_2\colon \G^\infty_v(U)\to  F_\varepsilon$ comes given by $p_2(\phi)=(0,\phi)$, with $\|k_\varepsilon\|_v\leq\varepsilon$ since 
$$
v(x)\left\|k_\varepsilon(x)\right\|
=v(x)\left\|(T_\varepsilon\circ p_2\circ \varepsilon\Delta_v)(x)\right\|
\leq v(x)\left\|T_\varepsilon\right\|\varepsilon\left\|\Delta_v(x)\right\|
\leq\varepsilon\left\|T_\varepsilon\right\|\leq\varepsilon
$$
for all $x\in U$. We have
\begin{align*}
(h_\varepsilon+k_\varepsilon)(x)&=T_\varepsilon(g_\varepsilon(x),0)+T_\varepsilon(0,\varepsilon \Delta_v(x))=T_\varepsilon(T_{g_\varepsilon}(\delta_x),\varepsilon\delta_x)\\
&=(T_\varepsilon\circ R_\varepsilon)(\delta_x)=(\iota_F\circ S_\varepsilon\circ R_\varepsilon)(\delta_x)\\
&=(\iota_F\circ T_f)(\delta_x)=(\iota_F\circ  f)(x)
\end{align*}
for all $x\in U$, and thus $h_\varepsilon+k_\varepsilon=\iota_F\circ f$. Hence $\|\iota_F\circ f-h_\varepsilon\|_v=\|k_\varepsilon\|_v\leq\varepsilon$, that is, $\iota_F\circ f\in\overline{\I^{\Hv^\infty}}(U,\ell_\infty(B_{F^*}))$ and thus $f\in(\overline{\I^{\Hv^\infty}})^{inj}(U,F)$.
\end{proof}

In the case that the weighted holomorphic ideal $\I^{\Hv^\infty}$ is equipped with a Banach ideal norm, Theorem \ref{t2} admits the following improvement.

\begin{corollary}\label{c8}
Let $[\I^{\Hv^\infty},\left\|\cdot\right\|_{\I^{\Hv^\infty}}]$ be a Banach weighted holomorphic ideal and let $f\in\Hv^\infty(U,F)$. The following are equivalent:
\begin{enumerate}
\item $f$ belongs to $(\overline{\I^{\Hv^\infty}})^{inj}(U,F)$.
\item There exists a complex Banach space $G$, a mapping $g\in\I^{\Hv^\infty}(U,G)$ and a function $N\colon\mathbb{R}^+\to\mathbb{R}^+$ such that
$$
\left\|\sum_{i=1}^n\lambda_i v(x_i)f( x_i)\right\|\leq N(\varepsilon)\left\|\sum_{i=1}^n\lambda_i v(x_i) g( x_i)\right\|+\varepsilon\sum_{i=1}^n\left|\lambda_i\right|
$$
for all $n\in\N$, $\lambda_1,\ldots,\lambda_n\in\C$, $ x_1,\ldots, x_n\in U$, and $\varepsilon>0$. 
\end{enumerate}
\end{corollary}

\begin{proof}
In view of Theorem \ref{t2}, all we need to show is $(i)\Rightarrow(ii)$. Let $f\in(\overline{\I^{\Hv^\infty}})^{inj}(U,F)$. By Theorem \ref{t2}, for each $m\in\N$, there are a complex Banach space $G_m$ and a map $g_m\in\I^{\Hv^\infty}(U,G_m)$ such that
$$
\left\|\sum_{i=1}^n\lambda_i v(x_i)f( x_i)\right\|\leq\left\|\sum_{i=1}^n\lambda_i v(x_i) g_m( x_i)\right\|+\frac{1}{2^m}\sum_{i=1}^n\left|\lambda_i\right|
$$
for all $n\in\N$, $\lambda_1,\ldots,\lambda_n\in\C$ and $ x_1,\ldots, x_n\in U$. Take the Banach space $G=(\oplus_{m\in\N}G_m)_{\ell_1}$ and, for each $m\in\N$, the canonical inclusion $I_m\colon G_m\to G$. Then $I_m\circ g_m\in\I^{\Hv^\infty}(U,G)$, and because of 
$$
\sum_{k=1}^m\frac{\left\|I_k\circ g_k\right\|_{\I^{\Hv^\infty}}}{2^k\left\|g_k\right\|_{\I^{\Hv^\infty}}}
\leq\sum_{k=1}^m\frac{\left\|I_k\right\|\left\|g_k\right\|_{\I^{\Hv^\infty}}}{2^k\left\|g_k\right\|_{\I^{\Hv^\infty}}}\leq \sum_{k=1}^m\frac{1}{2^k}\leq 1,
$$
for all $m\in\N$, the series $\sum_{m\geq 1}(I_m\circ g_m)/2^m\left\|g_m\right\|_{\I^{\Hv^\infty}}$ converges in the Banach space $(\I^{\Hv^\infty}(U,G),\left\|\cdot\right\|_{\I^{\Hv^\infty}})$ to the weighted holomorphic map $g=\sum_{m=1}^\infty(I_m\circ g_m)/2^m\left\|g_m\right\|_{\I^{\Hv^\infty}}\in\I^{\Hv^\infty}(U,G)$. Using the inequality above, we deduce 
\begin{align*}
\left\|\sum_{i=1}^n\lambda_i v(x_i)f( x_i)\right\|
&\leq 2^m\left\|g_m\right\|_{\I^{\Hv^\infty}}\left\|\sum_{i=1}^n\frac{\lambda_iv(x_i)}{2^m\left\|g_m\right\|_{\I^{\Hv^\infty}}} g_m( x_i)\right\|+\frac{1}{2^m}\sum_{i=1}^n\left|\lambda_i\right|\\
&\leq 2^m\left\|g_m\right\|_{\I^{\Hv^\infty}}\sum_{m=1}^\infty\left\|\sum_{i=1}^n\frac{\lambda_iv(x_i)}{2^m\left\|g_m\right\|_{\I^{\Hv^\infty}}} g_m( x_i)\right\|+\frac{1}{2^m}\sum_{i=1}^n\left|\lambda_i\right|\\
&=2^m\left\|g_m\right\|_{\I^{\Hv^\infty}}\left\|\sum_{i=1}^n\lambda_iv(x_i) g( x_i)\right\|+\frac{1}{2^m}\sum_{i=1}^n\left|\lambda_i\right|
\end{align*}
for all $n\in\N$, $\lambda_1,\ldots,\lambda_n\in\C$, $ x_1,\ldots, x_n\in U$ and $m\in\N$. Finally, this inequality yields the inequality in the statement defining $N\colon\mathbb{R}^+\to\mathbb{R}^+$ by 
$$
N(\varepsilon)=\left\{
\begin{array}[]{lll}
	2\left\|g_1\right\|_{\I^{\Hv^\infty}} &\text{if}& \varepsilon>1,\\
	& & \\
	2^m\left\|g_m\right\|_{\I^{\Hv^\infty}} &\text{if}& 2^{-m}<\varepsilon\leq 2^{-m+1},\; m\in\N.
\end{array}
\right.
$$
\end{proof}


\textbf{Author contributions.} Contributions from all authors were equal and significant. The original manuscript was read and approved by all authors.\\

\textbf{Funding}. This research was partially supported by Junta de Andaluc\'ia grant FQM194, and by grant PID2021-122126NB-C31 funded by MCIN/AEI/ 10.13039/501100011033 and by ``ERDF A way of making Europe''. Funding for open access charge: Universidad de Almer\'ia (Spain) / CBUA.\\

\textbf{Conflict of interest}. The authors have no relevant financial or non-financial interests to disclose.\\

\textbf{Data availability}. No data were used to support this study.\\

\textbf{Competing interests}. The authors declare no competing interests.



\end{document}